\documentclass[bj,authoryear]{imsart}

\usepackage{booktabs}
\usepackage{caption}
\startlocaldefs
\theoremstyle{plain}

\newtheorem{theorem}{Theorem}[section]
\newtheorem{lemma}[theorem]{Lemma}
\newtheorem{remark}{Remark}
\newtheorem{corollary}{Corollary}
\newtheorem{proposition}{Proposition}
\theoremstyle{remark}
\newtheorem{definition}[theorem]{Definition}
\newtheorem*{example}{Example}

\endlocaldefs

\begin{document}

\begin{frontmatter}
\title{Usual stochastic orderings of the second-order statistics with dependent heterogeneous semi-parametric distribution random variables}
\runtitle{Usual stochastic orderings of the second-order statistics}

\begin{aug}
\author[A]{\inits{F.}\fnms{}~\snm{Guoqiang   Lv}\ead[label=e1]{2022212109@nwnu.edu.cn}\orcid{0009-0009-4838-3681}}

\address[A]{Department of Mathematics and Statistics, Northwest Normal University, Lanzhou, China, Email:\printead[presep={\ }]{e1}}

\end{aug}

\begin{abstract}
This manuscript investigates the stochastic comparisons of the second-order statistics from dependent  and heterogeneous 
general semi-parametric
family of distributions observations. Some sufficient conditions on the usual stochastic order of the second-order statistics
from dependent and heterogeneous observations are established
under the p-larger order and the reciprocally majorization order. Some numerical
examples are given to illustrate the theoretical findings. In addition, the results of the Theorem are applied to two important models. Finally, we use a group of real data for empirical analysis to carry out reliability analysis.
\end{abstract}

\begin{keyword}
\kwd{Archimedean copula}
\kwd{p-larger order}
\kwd{reciprocally majorization}
\kwd{second-order statistic}
\kwd{stochastic order}
\kwd{semi-parametric family of distribution}

\end{keyword}

\end{frontmatter}

\section{Introduction}

Order statistics have been gained significant attention for their important role in statistics, reliability theory, life testing, survival analysis, auction theory, operations research, and many other application areas. Let $ X _ { 1 : n } \leq X _ { 2 : n } \leq \cdots \leq X _ { n : n }$ be the order statistic generated by the random sample $ X _ { 1 } , X _ { 2 } , \ldots , X _ { n }$. It is widely recognized that the $k$th  order statistic $ X _ { k : n } ( k = 1 , 2 , \ldots , n )$ corresponds to the lifetime of a $ ( n - k + 1 )$ -out-of-$n$ system in reliability theory, which is a very popular redundancy structure in fault-tolerant systems and has been employed in many industrial and military systems. 
 Concretely, the lifetimes of parallel and series systems correspond to the order statistics $X _ { n:n }$ and $ X _ { 1:n }$, respectively.
 
Stochastic comparisons of extreme order statistics have been studied extensively over the past few decades. There have been a considerable number of scholarly studies on independent extreme order statistics. For example, \cite{proschan1976stochastic} studied the usual stochastic order of maximal order statistics with independent heterogeneous exponential samples. \cite{zhao2016extreme} established the usual stochastic, hazard rate, and likelihood ratio orderings of extreme order statistics from heterogeneous independent and interdependent Weibull samples. \cite{fang2018ordering} build the usual stochastic order from minimums of two scale proportional hazard samples with Archimedean survival copulas, as well as deriving the hazard rate order between minimums of independent scale proportional hazard (SPH) samples and the reversed hazard rate order between maximums of independent scale proportional reversed hazard (SPRH) samples.
\cite{chen2019comparisons} studied the usual stochastic order and the corresponding properties of extreme order statistics arising from independent negative binomial random variables.  
Dependent data is often encountered in the fields of finance, biomedicine, geological research and so on(cf.\cite{chang2024central}). Due to the complex dependencies between different components and the nonlinear dynamic behavior in this series, there are great challenges in reliability analysis for this type of data.There have also been a substantial number of studies on the extreme order statistic of dependence.
\cite{kochar2015stochastic} studied the hazard rate order of extreme order statistics from two sets of dependent and heterogeneous scale (SC) observations. \cite{torrado2021allocation} investigated the stochastic comparison of series-parallel systems under usual stochastic order. \cite{fang2020optimal} investigated the stochastic comparison of parallel–series system and series-parallel systems under usual stochastic, hazard rate and reversed hazard rate order. \cite{hart2022posterior} studied various order relationships between posteriori under different distributions. \cite{das2022ordering} studied the usual stochastic, hazard rate and reversed hazard rate orderings of extreme order statistics of heterogeneous dependent random variables under a class of wide parametric model including the modified proportional hazard rate  scale (MPHRS) and the modified proportional reversed hazard rate scale (MPRHRS) models. \cite{barmalzan2020stochasticc} studied the usual stochastic, star and convex transform orders of both series and parallel systems comprising heterogeneous and dependent extended exponential components under Archimedean copula. \cite{das2022ordering} established
the usual stochastic and hazard rate orderings for the smallest claim amount from two sets of portfolios of risks with heterogeneous dependent exponentiated location-scale (ELS) claims. 
For more comprehensive references on extreme order statistics, one can refer to 
\cite{yan2013further}, \cite{li2016stochastic} and \cite{fang2018ordering}.

Second-order statistics include the second smallest order statistic $(X _ { 2:n })$ and the second largest order statistic $(X _ { n-1:n })$. The research on second-order statistics has not only been a popular research trend in recent years, but it also has a extensive and significant background of applications. The lifetime of the $ (n-1)$ -out-of-$n$ corresponds to the second smallest order statistic $(X _ { 2:n })$, and the system is also called a fail-safe system. In reliability and safety engineering, the fail-safe system is one of commonly used fault-tolerant systems and refers to the situation where the failure of a single component cannot lead to system failure, but the failures of more than two components cause system failed (cf.\cite{barlow1996mathematical}). Hence, the fault tolerance of a fail-safe system has been a fundamental design attributed to achieving much more higher reliability, which demands redundancy in the design of functioning systems (cf.\cite{pham1992reliability,chandra1997reliability}). Some typical examples include the space shuttle and nuclear power plant control, which are designed to remain safe in the event of a failure, rather than not designed to prevent failure (cf.\cite{lala1985fault}). Applications of such systems can be found in the safety monitoring and reactor trip function systems (cf.\cite{systelils1975ieee}). For example, the nuclear plants often use three functioning counters to monitor the radioactivity of the air in ventilation systems, intending to initiate reactor shutdown when a dangerous level of radioactivity is present, and when two or more counters register a dangerous level of radioactivity, the reactor automatically shuts down (cf.\cite{pham1992reliability}). In addition to the above background on the application of second-order statistics, another one is auction theory. $X _ { 2:n }$ denotes the winner’s price for the bid in the second-price reverse auction. $X _ { n-1:n }$ denotes a sealed-bid second-price auction, also known as a Vickrey auction (cf.\cite{paul2004mean,cox1982theory}).  

In recent years there have also been a number of scholarly studies on second-order statistics, mainly focusing on stochastic comparisons of different models under different the majorizations order. For instance, \cite{pualtuanea2008comparison} considered the hazard rate order of the second-order order statistics under independent exponential models. 
\cite{zhao2009characterization} extended \cite{pualtuanea2008comparison}’s results of the hazard rate order
to the mean residual time order. \cite{zhao2011dispersive} generalized \cite{pualtuanea2008comparison}’s results of the exponential distribution to the proportional hazard rate (PHR) model. \cite{balakrishnan2015improved} established the mean residual life, dispersive, hazard rate, and likelihood ratio orderings of the second-order statistics arising from two sets of independent exponential observations.  \cite{fang2016stochastic} conducted the second smallest
(largest) order statistic from proportional hazard rate (PHR) and the proportional reversed hazard rates (PRHR) models, and also obtained some results under the extreme value order statistics. Additionly, \cite{li2016stochastic} obtained the usual stochastic order of the sample extremes and the second
smallest order statistic from the scale (SC) model. \cite{wang2021testing} presents a two-sample nonparametric goodness of fit test method for uniform random ordering. \cite{yan2023stochastic} studies the stochastic comparisons of the second order statistics from dependent or independent and heterogeneous
modified proportional hazard rate (MPHR) models. \cite{shrahili2023relative} identify several relative (reversed) hazard rate order properties of the modified proportional hazard rate (MPHR) and modified proportional reversed hazard rate (MPRHR) models and investigation is to see how a relative ordering between two possible base distributions for the response distributions in these models is preserved when the parameters of the underlying models are changed. 
\cite{shojaee2024ordering} provided sufficient conditions to
compare the smallest and the second smallest (largest and second largest) order statistics of dependent
and heterogeneous random variables having the additive hazard (AH) model with the Archimedean copula in the sense of usual stochastic order and hazard rate order. \cite{das2024stochastic} considered
stochastic comparisons between second-order statistics arising from general exponentiated location scale (ELS) models when the random variables are independent, and established usual stochastic and hazard
rate orders between second-order statistics. \cite{sahoo2024multivariate} studies the developed sequential order statistics and developed generalized order statistics models that incorporate dependency structures among ordered random vectors.
In addition, in order to improve the performance of the reliability system, active redundancy technology and random shock model are widely used in reliability engineering. For a systematic and comprehensive discussion of this section, the reader can refer to \cite{bian2023reliability};\cite{wei2023reliability};\cite{ding2023criticality}.

In the above literature review, most scholars mainly study the random comparison of second-order statistics for specific models in the family of semi-parametric distributions with parameter vectors satisfying different majorizations orders. However, most of the models they study are specific models under the family of semi-parametric distribution, and there are very few second-order statistics for the family of semi-parametric distribution. In fact, the semi-parametric model is elaborated as follows: If $ d _ { \boldsymbol{\theta} } : \left[ 0 , 1 \right] \rightarrow \left[ 0 , 1 \right]$ is a
decreasing function for which $ d _ { \boldsymbol{\theta} } ( 0 ) = 0 , d _ { \boldsymbol{\theta} } ( 1 ) = 1$ where $ \boldsymbol{\theta} = ( \theta _ { 1 } ,\theta _ { 2 }, \ldots , \theta _ { n } ) \in\Theta$ is a vector of unknown parameters, and $\Theta$ is the parameter space. Then, the random variable
$X$ with the following cumulative distribution function$$ F ( x ; \boldsymbol{\theta} ) = d _ { \boldsymbol{\theta} } ( F ( x ) ) ,$$
is said to have a semi-parametric model where $F(t)$ is the baseline distribution.

Very recently, \cite{hazra2024ordering} studied the weakly super-majorization and  weakly sub-majorizeation of parameter vectors under second-order order statistics of a general family of semi-parametric distributions. Moreover, there are few random comparisons of second-order statistics between the p-larger and the reciprocally majorization in the majorization order of parameter vectors. For detailed studies of the p-larger and the reciprocally majorization in the extreme order statistics, please refer to \cite{das2022ordering}. Motivated by the work \cite{hazra2024ordering,das2022ordering}, in this paper, we consider two samples dependent and heterogeneous in a family of semi-parametric distributions. Sufficient conditions are established to randomly compare the second-order statistics of the p-larger and the reciprocally majorization of parameter vectors in two samples in the sense of  the usual stochastic order.
The contribution and innovation of this manuscript are reflected in the following aspects.
\begin{itemize}
\item Most scholars stochastic comparison of order statistics is mainly based on majorization, super-majorization, sub-majorization, but there are few researches on order statistics in the p-larger and the reciprocal majorization orders;
\item The existing literature mainly focuses on the problem of random comparison under specific models, which in fact belong to the family of semi-parametric distributions. In this paper, random comparison in a family of semi-parametric distributions is considered;    
\item It is different from most existing studies on analyzing the reliability of fail-safe systems with independent components. The real world is complex, things are more or less connected to each other, and we use the very useful copula to model the dependent structure between the lifetime of the system components.    
\end{itemize}	

The remaining part of the paper is organized as follows. Section \ref{pre} recalls some basic concepts and notations that will be used in the sequel. Section \ref{lar}  establishes the usual stochastic
order of the second smallest statistics from dependent and heterogeneous observations in the parameter vectors satisfy the p-larger or the reciprocally majorization. Section \ref{corollary} considers the problem of random comparison under two specific models.
Section \ref{data} demonstrate some of the derived results
with real data.
Section \ref{conclud} concludes this paper and future directions.

\section{Preliminaries}\label{pre}
In this section, we briefly present definitions of some well-known concepts relating to stochastic orders, majorization orders and copulas.Throughout the article, the term ``increasing'' and ``decreasing'' are used in a non-strict  sense.Assume  $X$  is an non-negative random variable,  denote the  distribution function, survival function, probability density function, hazard rate function and reversed hazard rate function by $F_{X}\left(t\right)$, $\bar{F}_{X}(t) = 1-F_{X}(t)$,   $f_X(t)$,  $h _ { X }(t) = f _ { X }(t) / \bar{F} _ { X }(t)$ , and  $\tilde{r}_ { X }(t) = f _ { X }(t) / {F} _ { X }(t)$, respectively. In addition, for convenience, we use $a\overset{ \rm sgn}{=} b$
 to denote that both sides of the equality have the same sign.

Stochastic order is a very useful tool to compare random variables arising from reliability theory, operations research, actuarial science,  economics, finance, and so on.

We begin this subsection with the following definition of stochastic orders.

\begin{definition}
	A random variable	 $X$ is said to be smaller than $Y$ in the
	\begin{enumerate} 
		\item usual stochastic order (denoted by $X \leq_{st} Y$) if $\bar{F} _ { X }(t) \leq \bar{F} _ { Y }(t)$ for all $t \geq 0$; 
\item hazard rate order (denoted by $X \leq_{hr} Y$) if $\bar{F} _ { Y }(t) / \bar{F} _ { X }(t)$ is increasing in $t \geq 0$. It's equivalent to verifying that $h _ { X }(t)\geq h _ { Y }(t)$;
\item reversed hazard rate order (denoted by $X \leq_{rh} Y$) if ${F} _ { Y }(t) / {F} _ { X }(t)$ is increasing in $t \geq 0$. It's equivalent to verifying that $\tilde{r} _ { X }(t)\leq \tilde{r} _ { Y }(t)$.
	\end{enumerate}	
\end{definition}
It is well known that the reversed hazard rate order or the hazard rate order both implies the usual
stochastic order, but the reversed statement is not necessarily true in general. For more on stochastic orders and their properties, please see monographs of \cite{shaked2007stochastic}, \cite{li2013stochastic}.

Next we introduce the notion of majorization which is one of the basic tools in establishing various inequalities in statistics and probability.

\begin{definition}
Let vector $\boldsymbol{a} = \left( a _ { 1 } , \ldots , a _ { n } \right)$ and vector $\boldsymbol{b} = ( b _ { 1 } , \ldots , b _ { n } )$ be with increasing arrangements $ a _ { 1:n } \leq ... \leq a _ { n:n } $ and $ b _ { 1:n } \leq ... \leq b _ { n:n } $ , respectively. Then,
	\begin{enumerate} 
\item $\boldsymbol{a}\in \mathbf{R}^n$ is said to majorize $\boldsymbol{b}\in \mathbf{R}^n$ (denoted by $\boldsymbol{a} \overset{m}{\succeq} \boldsymbol{b}$) if $\sum _ { j = 1 } ^ { i } a _ {  j : n } \leq \sum _ { j = 1 } ^ { i } b _ {  j: n  } $ and 
    $\sum _ { j = 1 } ^ { n } a _ {  j:n  } = \sum _ { j = 1 } ^ { n } b _ { j:n } $ for all $i = 1,2,...,n-1$;
\item $\boldsymbol{a}\in \mathbf{R}^n$ is said to weakly super-majorize $\boldsymbol{b}\in \mathbf{R}^n$ (denoted by $\boldsymbol{a} \overset{w}{\succeq} \boldsymbol{b}$) if $\sum _ { j = 1 } ^ { i } a _ {  j : n } \leq \sum _ { j = 1 } ^ { i } b _ {  j: n  } $ for all $i = 1,2,...,n$;    
\item $\boldsymbol{a}\in \mathbf{R}^n$ is said to weakly sub-majorize $\boldsymbol{b}\in \mathbf{R}^n$ (denoted by $\boldsymbol{a} \underset{w}{\succeq} \boldsymbol{b}$) if $\sum _ { j = 1 } ^ { i } a _ {  j : n } \geq \sum _ { j = 1 } ^ { i } b _ {  j: n  } $ for all $i = 1,2,...,n$;    
\item $\boldsymbol{a}\in \mathbf{R}_{+}^n$ is said to p-larger $\boldsymbol{b}\in \mathbf{R}_{+}^n$ (denoted by $\boldsymbol{a} \overset{p}{\succeq} \boldsymbol{b}$) if $ \prod _ { j = 1 } ^ { i } a _ { j:n } \leq \prod _ { j = 1 } ^ { i } b _ { j:n } $ for all $i = 1,2,...,n $;
\item $\boldsymbol{a}\in \mathbf{R}_{+}^n$ is said to reciprocal majorize $\boldsymbol{b}\in \mathbf{R}_{+}^n$ (denoted by $\boldsymbol{a} \overset{rm}{\succeq} \boldsymbol{b}$) if $\sum _ { j = 1 } ^ { i } \frac { 1 } { a _ { j :n } } \geq \sum _ { j = 1 } ^ { i } \frac { 1 } { b _ { j : n } } $ for all $i = 1,...,n $.
	\end{enumerate}	
\end{definition}
It is worth noting that, 
$$\boldsymbol{a} \overset{w}{\preceq} \boldsymbol{b}\Leftarrow\boldsymbol{a} \overset{m}{\preceq} \boldsymbol{b} \Rightarrow\boldsymbol{a} \underset{w}{\preceq} \boldsymbol{b} \Rightarrow \boldsymbol{a} \overset{p}{\preceq}\boldsymbol{b} \Rightarrow \boldsymbol{a} \overset{rm}{\preceq} \boldsymbol{b} , $$
for any two real-valued vectors $\boldsymbol{a}$ and $\boldsymbol{b}$ , while the reverse is not true in general.

For more details on majorization, p-larger, and reciprocal majorization orders and their applications, please see \cite{marshall1979inequalities} and \cite{zhao2009characterization}.

	Next, let us review the concept of copula. For a random vector $\bm{X}=(X_1,X_2,\ldots,X_n)$ with the joint distribution function  $ H(\bm{x})$ and respective marginal distribution functions $F_1,F_2, \ldots,F_n$,  the \emph{copula} of $X_1,X_2,\ldots,X_n$ is a distribution function $C:[0,1]^n\mapsto [0,1]$ satisfying $$H(\bm{x})=\mathbb{P}(X_1\le x_1,X_2\le x_2,\ldots,X_n\le x_n)=C(F_1(x_1),F_2(x_2),\ldots,F_{n}(x_n)).$$  Similarly, a \emph{survival copula} of $X_1,X_2,\ldots,X_n$  is a survival function $\hat{C} :[0,1]^n\mapsto [0,1]$ satisfying
\begin{equation*}
\bar{H}(\bm{x})=\mathbb{P}(X_1>x_1,X_2>x_2,\ldots,X_n>x_n)=\hat{C}(\bar{F}_1(x_1), \bar{F}_2(x_2),\ldots,\bar{F}_n(x_n)),
\end{equation*}
where $\bar{H}(\bm{x})$ is the joint survival function.

Copulas are very efficient functions to describe the dependence for two or more dependent random variables. It has many applications in the field of finance and reliability. 
There are many types of copula function, such as t-copula, Gaussian copula, the Archimedean copula and FGM copula. More copula functions can be seen in \cite{nelsen2006introduction}.
Next we mainly introduce definition of the Archimedean copula.
\begin{definition}{\rm\citep[Expression 4.6.1 on Page 151 of][]{nelsen2006introduction}}\label{dycopula}Let $\psi :[ 0 , +\infty ) \mapsto [ 0 , 1 ]$ be a decreasing and continuous function with $\psi(0) = 1$  and $\psi(+\infty) = 0 $, and let $\phi = \psi^{-1}$ be the pseudo-inverse of $\psi$. Then, the
Archimedean copula is defined as
$$C _ { \psi } ( u _ { 1 } , \cdots , u _ { n } ) = \psi \Big(\sum_{i=1}^n\phi (u_i)\Big) , u _ { i } \in [ 0 , 1 ] , i = 1 , 2,\cdots , n.$$
where $\psi$ is the generator satisfying the conditions:  $( - 1 ) ^ { k } \psi ^ {  (k)  } ( x ) \geq 0 $, for $k = 0,1,...n-2,$ and $( - 1 ) ^ { n-2 } \psi ^ {  (n-2)  } ( x )$  is decreasing and convex.
\end{definition}
Presently, the Archimedean copula is a kind of widely used function because of their nice properties such as simple form, symmetry and the ability of combining. 
It includes the Frank copula, the Gumbel copula and the Clayton copula and so on.

Next, we formulate three lemmas which are used to prove the main results.
\begin{lemma}{\rm\citep[A.4.Theorem on Page 84 of][]{marshall1979inequalities}}\label{yl1}
For an open interval $I \subseteq \mathbb{R}$, a
continuously differentiable $h : I^{n} \rightarrow \mathbb{R}$ is {\rm {Schur-convex(resp. Schur-concave)}} if and only if $h$ is symmetric on $I^{n}$ and
$$\left( x _ { i } - x _ { j }\right ) \left( \frac { \partial h ( x ) } { \partial x _ { i } } - \frac { \partial h ( x ) } { \partial x _ { j } } \right) \geq(resp. \leq) 0 ,\ fo r\ a l l\ 1 \leq i \neq j \leq n\ a n d\ x \in I ^ { n } . $$
\end{lemma}
\begin{lemma}{\rm\citep[Lemma 2.1 of][]{khaledi2002dispersive}}\label{yl2}
 Let $S \subseteq \mathbb{R}_ { + } ^ { n }$ and $h : S \rightarrow \mathbb{R}$ be a function. Then, for $\boldsymbol{x},\boldsymbol{y} \in S$
    $$ \boldsymbol{x}\overset{p}{\succeq} \boldsymbol{y} \Rightarrow h(\boldsymbol{x}) \geq(resp. \leq) h(\boldsymbol{y}) $$ if and only if
\begin{enumerate}
  \item $ h(e^{a_{1}},...,e^{a_{n}})$is {\rm {Schur-convex(resp. Schur-concave)}} in $(a _ { 1 } , \cdots , a _ { n }  ) $,
\item $ h(e^{a_{1}},...,e^{a_{n}})$ is {\rm {decreasing(resp. increasing)}} in $a _ { i }$ for $ i=1,...,n $, where $a _ { i } = \log x_{i}$  for $i = 1,2,...,n$.
\end{enumerate}
\end{lemma}
\begin{lemma}{\rm\citep[Lemma 1 of][]{hazra2017stochastic}}\label{yl3}
 Let $S \subseteq \mathbb{R}_ { + } ^ { n }$ and $h : S \rightarrow \mathbb{R}$ be a function. Then, for $\boldsymbol{x},\boldsymbol{y} \in S$
    $$ \boldsymbol{x}\overset{rm}{\succeq} \boldsymbol{y} \Rightarrow h(\boldsymbol{x}) \geq(resp. \leq) h(\boldsymbol{y}) $$ if and only if
\begin{enumerate}
  \item $ h(\frac{1}{a_{1}},...,\frac{1}{a_{n}})$is {\rm {Schur-convex(resp. Schur-concave)}} in $(a _ { 1 } , \cdots , a _ { n } ) $,
\item $ h(\frac{1}{a_{1}},...,\frac{1}{a_{n}})$ is {\rm {increasing (resp. dereasing)}} in $a _ { i }$ for $ i=1,...,n $, where $a _ { i } = \frac{1}{x_{i}}$  for $i = 1,2,...,n$.
\end{enumerate}
\end{lemma}

\section{Usual stochastic order of the second-order statistics}\label{lar}
In this section, we compare the second-order statistic of two vectors of dependent and heterogeneous random variables in the sense of
usual stochastic order. The heterogeneity and the structure dependency of the random variables follows from a general semi-parametric family of distributions under the Archimedean copula. 

Suppose that $\boldsymbol{X}=(X_{1},X_{2},..., X_{n})$ and $\boldsymbol{Y}=(Y_{1},Y_{2},..., Y_{n})$ to be two vector of dependent and heterogeneous random
variables having lifetimes following the Semi-parametric model with the survival function $\bar{F}  {  } ( x ; \theta _ { i } )$, and $\bar{F}  {  } ( x ; \theta _ { i } ), i=1,2,...,n$, respectively. The survival function of the second smallest order statistic of the dependent and heterogeneous random variables with the parameters vectors $\boldsymbol\theta = (\theta_{1},...,\theta_{n})$ and $\boldsymbol\theta^{\ast} = (\theta_{1}^{\ast},...,\theta_{n}^{\ast})$ under common the Archimedean copula are, respectively, as follows:
$$ \bar{ F } _{{X} _ { 2:n }} ( x ) = \sum _ { i = 1 } ^ { n } \psi \left[ \sum _ { j \neq i } \phi \left( \bar { F } ( x ; \theta _ { j }  ) \right) \right] - ( n - 1 ) \psi \left[ \sum _ { i = 1 } ^ { n } \phi \left( \bar{ F } ( x ; \theta _ { i } ) \right) \right] ,$$
$$ \bar { F }_{{Y} _ { 2:n }} ( x ) = \sum _ { i = 1 } ^ { n } \psi \left[ \sum _ { j \neq i } \phi \left( \bar { F } ( x ; \theta _ { j } ^{\ast} ) \right) \right] - ( n - 1 ) \psi \left[ \sum _ { i = 1 } ^ { n } \phi \left( \bar { F } ( x ; \theta _ { i }^{\ast} ) \right) \right].$$
In the following Theorem, we compare the second smallest order statistics when the parameter vectors verify the p-larger order with common generator of Archimedean copulas in the sense of usual stochastic order.
\begin{theorem}\label{th11}
	Let $\boldsymbol{X}$ and $\boldsymbol{Y}$ be two vectors of nonnegative dependent random variables with Archimed\\ean copula of common generator $\psi.$ Assume  $X _ {i} \sim F  {  } ( x ; \theta _ { i } ) $ and $Y _ {i} \sim F  {  } ( x ; \theta _ { i }^{\ast} ) $ for $i=1,2,...,n$. If the following conditions hold:
\begin{enumerate}
  \item $\psi$ is log-concave\rm{;}
  \item $\bar { F } ( x ; e^{a}  )$ is decreasing and log-convex  in $a$\rm{;}
\end{enumerate}
Then, $\boldsymbol\theta \overset{p}{\succeq}  \boldsymbol\theta^{\ast}$ implies $X _ { 2:n } \geq_ { s t } Y _ { 2 : n }$.
\end{theorem}
\begin{proof}
To obtain the desired result, it is required to show that
$$\bar { F }_{{X} _ { 2:n }} ( x ) \geq \bar{ F }_{{Y} _ { 2:n }} ( x ).$$

To do this, let us define
\begin{equation}
\ {h}_ { 1 }(\boldsymbol{\theta}) = \sum _ { i = 1 } ^ { n } \psi \left[ \sum _ { j \neq i } \phi \left( \bar { F } ( x ; \theta _ { j }  ) \right) \right] - ( n - 1 ) \psi \left[ \sum _ { i = 1 } ^ { n } \phi \left( \bar{ F } ( x ; \theta _ { i } ) \right) \right].
\end{equation}

Thus, from Lemma \ref{yl2}, it is required to prove that ${h}_ { 1 }(e^{a_{1}},...,e^{a_{n}})$ is decreasing in $a_{i}$, for all $i = 1,..., n$, and it is Schur-convex in $\boldsymbol{a} = ( a _ { 1 } , \ldots , a _ { n } )$,where $a_{i} = \log\theta _ { i }$, for $i = 1,...,n.$ Taking the partial derivative of ${h} _ { 1 }(\boldsymbol{e^{a}})$ with respect $a_{i}$, where $\boldsymbol{e^{a}}=(e^{a_{1}},...,e^{a_{n}})$,we obtained
\begin{align*}
  \frac { \partial {h} _ { 1 } ( \boldsymbol{e^{a}}) } { \partial a _ { p } }
  &= \left\{\sum _ { i \neq p }  \psi^{\prime} \left[ \sum _ { j \neq i } \phi \left( \bar { F } ( x ; e^{a_{j}}  ) \right) \right] - ( n - 1 ) \psi^{\prime} \left[ \sum _ { i = 1 } ^ { n } \phi \left( \bar { F } ( x ; e^{a_{i}} ) \right) \right] \right\}\\
  &\quad \quad\times\left(\frac { \partial\left( \bar { F } ( x ; e^{a_{p}}  ) \right) } { \partial a _ { p } }\right) \frac{1}{\psi^\prime(\phi( \bar { F } ( x ; e^{a_{p}}  ))}\\
  &\leq 0,
\end{align*}
which holds because  $\bar { F } ( x ; e^{a_{i}}  )$  is decreasing in $ a_{i} $ and the following equation $\delta _ {  } \left( z , x \right)$ is decreasing in z > 0.
Actually,for $1 \leq p \neq k \leq n $, let
\begin{align*}
  \delta _ {  } \left( z , x \right)
  &= ( n - 1 ) \psi ^ { \prime } \left[ \sum _ { i = 1 } ^ { n } \phi ( \bar{ F } ( x ; e^{a_{i}} ) ) \right] - \sum _ { i \notin \left\{ p , k \right\} } \psi ^ { \prime } \left[ \sum _ { j \neq i } \phi ( \bar { F } ( x ; e^{a_{j}} ) ) \right] \\
   &\quad \quad- \psi ^ { \prime } \left[ \phi ( \bar { F } ( x ; e^{ z } ) ) + \sum _ { j \notin \left\{ p , k \right\} } \phi ( \bar{ F } ( x ; e^{a_{j}} ) ) \right].
\end{align*}

According to Definition \ref{dycopula}, it holds that $\psi ^ { \prime } ( x ) \leq 0 $ for $x \geq 0$ and $\psi ^ { \prime }$ is increasing.
Since $ \phi ( \bar { F } ( x ; e^{ z } ))$ is increasing in $z > 0$, the function $\delta _ {  } ( z , x )$ is decreasing in $z > 0$. In view
of $\phi ( x ) \geq 0 $ for $x \in \left[ 0 , 1 \right] $, we have

$$\psi ^ { \prime } \left[ \sum _ { i = 1 } ^ { n } \phi ( \bar { F } ( x ; e^{a_{i}} ) ) \right] \geq \psi ^ { \prime } \left[ \sum _ { j \neq i } \phi ( \bar { F } ( x ; e^{a_{i}} ) ) \right] , $$
and then
$$( n - 1 ) \psi ^ { \prime } \left[ \sum _ { i = 1 } ^ { n } \phi ( \bar { F } ( x ; e^{a_{i}} ) ) \right] \geq \sum _ { i \neq p } \psi ^ { \prime } \left[ \sum _ { j \neq i } \phi ( \bar { F } ( x ; e^{a_{i}} ) ) \right] . $$

As a result, for any $x \geq 0$, it holds that
\begin{align}
  \delta _ {  } ({a_{p}}, x )
  & = ( n - 1 ) \psi ^ { \prime } \left[ \sum _ { i = 1 } ^ { n } \phi ( \bar{ F } ( x ; e^{a_{i}} ) ) \right] - \sum _ { i \notin \left\{ p , k \right\} } \psi ^ { \prime } \left[ \sum _ { j \neq i } \phi ( \bar { F } ( x ; e^{a_{j}} ) ) \right] -  \sum _ { j \neq k} \phi \left( \bar{ F } ( x ; e^{a_{i}} ) \right) \nonumber\\
  &= ( n - 1 ) \psi ^ { \prime } \left[ \sum _ { i = 1 } ^ { n } \phi ( \bar { F } ( x ; e^{a_{i}} ) ) \right] - \sum _ { i \neq p } \psi ^ { \prime } \left[ \sum _ { j \neq i } \phi ( \bar { F } ( x ; e^{a_{j}} ) ) \right] \nonumber \\
  & \geq 0.
  \nonumber
\end{align}

Therefore
$$\sum _ { i \neq p }  \psi^{\prime} \left[ \sum _ { j \neq i } \phi ( \bar { F } ( x ; e^{a_{j}}  ) ) \right] - ( n - 1 ) \psi^{\prime} \left[ \sum _ { i = 1 } ^ { n } \phi ( \bar { F } ( x ; e^{a_{i}} ) ) \right] \leq 0 ,$$
and for $a _ { p } \geq a _ { q },$ we have $ \delta _ {  } ({a_{p}}, x ) \leq \delta _ {  } ({a_{q}}, x ) $, in other words
\begin{flalign}\label{e2}
  &\sum _ { i \neq q } \psi ^ { \prime } \left[ \sum _ { j \neq i } \phi ( \bar { F } ( x ; e^{a_{j}} ) ) \right]- ( n - 1 ) \psi^{\prime} \left[ \sum _ { i = 1 } ^ { n } \phi ( \bar { F } ( x ; e^{a_{i}} ) ) \right] \nonumber \\
  &\leq \sum _ { i \neq p } \psi ^ { \prime } \left[ \sum _ { j \neq i } \phi ( \bar { F } ( x ; e^{a_{j}} ) ) \right]-( n - 1 ) \psi^{\prime} \left[ \sum _ { i = 1 } ^ { n } \phi ( \bar { F } ( x ; e^{a_{i}} ) ) \right] \\
  &\leq 0.
  \nonumber
\end{flalign}

Since $\psi$ is log-concave  and $\bar{F} ( x ; e^{a_{i}} )$ is log-convex in $a_ { i }$ 
, for $a _ { p } \geq a _ { q } $,
\begin{equation}\label{e3}
  \frac { \psi ( \phi ( \bar { F } ( x , e^{a_{q}} ) ) ) } { \psi ^ { \prime } ( \phi ( \bar { F } ( x ; e^{a_{q}} ) ) }\leq\frac { \psi ( \phi ( \bar { F } ( x ;e^{a_{p}} ) ) } { \psi ^ { \prime } ( \phi ( \bar { F } ( x ; e^{a_{p}} ) ) } \leq 0,
\end{equation}
and
\begin{equation}\label{e4}
    {\frac { \partial( \bar { F } ( x ; e^{a_{q}}  ) ) } { \partial a _ { q } }} \frac { 1 } { \bar { F } ( x ;  e^{a_{q}} ) }\leq{\frac { \partial( \bar { F } ( x ; e^{a_{p}}  ) ) } { \partial a _ { p } }}\frac { 1 } { \bar { F } ( x ;  e^{a_{p}} ) } \leq 0.
\end{equation}

From (\ref{e3}) and (\ref{e4}) ,we get
\begin{equation}\label{e5}
 \frac { \partial( \bar { F } ( x ; e^{a_{q}}  ) ) } { \partial a _ { q } }\frac{1}{\psi^\prime(\phi( \bar { F } ( x ; e^{a_{q}}  ))}\geq\frac { \partial( \bar { F } ( x ; e^{a_{p}}  ) ) } { \partial a _ { p } }\frac{1}{\psi^\prime(\phi( \bar { F } ( x ; e^{a_{p}}  ))} \geq 0.
\end{equation}

For any $p \neq q$, we have
\begin{flalign}\label{e6}
  &\frac { \partial {h} _ { 1 } ( \boldsymbol{e^{a}}) } { \partial a _ { p } } - \frac { \partial {h} _ { 1 } ( \boldsymbol{e^{a}}) } { \partial a _ { q } } \nonumber\\
  &=\left\{\sum _ { i \neq p }  \psi^{\prime} \left[ \sum _ { j \neq i } \phi ( \bar { F } ( x ; e^{a_{j}}  ) ) \right] - ( n - 1 ) \psi^{\prime} \left[ \sum _ { i = 1 } ^ { n } \phi ( \bar { F } ( x ; e^{a_{i}} ) ) \right] \right\} \frac{\frac { \partial( \bar { F } ( x ; e^{a_{p}}  ) ) } { \partial a _ { p } }}{\psi^\prime(\phi( \bar { F } ( x ; e^{a_{p}}  ))} \\
  &-\left\{\sum _ { i \neq q }  \psi^{\prime} \left[ \sum _ { j \neq i } \phi ( \bar{ F } ( x ; e^{a_{j}}  ) ) \right] - ( n - 1 ) \psi^{\prime} \left[ \sum _ { i = 1 } ^ { n } \phi ( \bar { F } ( x ; e^{a_{i}} ) ) \right] \right\} \frac{\frac { \partial( \bar { F } ( x ; e^{a_{q}}  ) ) } { \partial a _ { q } }}{\psi^\prime(\phi( \bar { F } ( x ; e^{a_{q}}  ))}. \nonumber
\end{flalign}
On using (\ref{e2}) and (\ref{e5}) in (\ref{e6}), we get
$$( a _ { p } - a _ { q } )\left( \frac { \partial {h} _ { 1 } ( \boldsymbol{e^{a}}) } { \partial a _ { p } } - \frac { \partial {h} _ { 1 } ( \boldsymbol{e^{a}}) } { \partial a _ { q } } \right) \geq 0 .$$
Note that $\bar{ F } _{{X} _ { 2:n }}(x)$ is symmetric with respect to $a_{ i }$’s. Thus, by using Lemma \ref{yl2}, we get that $\bar{ F } _{{X} _ { 2:n }}(x)$ is Schur-convex in
$(a_{ 1 },...,a_{ n })$. Hence the result is proved.
\end{proof}
Theorem \ref{th11} shows that more heterogeneity among general semi-parametric family of distributions parameter vectors in the sense
of the p-larger order results in larger reliability under the usual stochastic order
in reliability theory.
\begin{remark}
Condition (i) of Theorem \ref{th11} says that $\psi$ is log-concave, which is not a complicated requirement. Next, we show that something satisfying condition (i) can be found, and give a further explanation of $\psi$ is log-concave. 
\begin{enumerate}
  \item  There exist many well-known generator of Archimedean copulas which satisfy the log-concavity of $\psi$.In Table \ref{tab1}, we shall give a summary  log-concave of generator of an Archimedean copula.
\begin{table}[!h]
		\centering 
		\setlength{\belowcaptionskip}{0.4cm}
		\caption{Summary log-concave of generator $\psi$ of an Archimedean copula}
		\label{tab1}
		\begin{tabular}{ccc}
			\toprule 
			Copula  & Generator & Parameter  \\
			\midrule 
            Gumble–Barnett & $\exp\left(\frac { 1 } { \theta } ( 1 - e ^ { x } ) \right)$ & $\theta \in (0,1]$ \\
            Gumbel-Hougaard & $\exp\left(1 - ( 1 + x ) ^ { \theta }\right)$ & $\theta \in (1,\infty)$ \\
            Ali-Mikhail-Haq & $ (1 - \theta)/ (\exp (t) - \theta)$ & $\theta \in [ -1 , 0 ]$
            \\
            
			\bottomrule
		\end{tabular}
	\end{table}
        \item $\psi$ is log-concave implies that the dependency structure between random variables is negative. First, we provide the definition of Reverse Regular of order 2 $(RR_2)$ (cf.\cite{karlin1980classes}): a function $\psi(u, v)$ is $RR_2$ if $\psi(u, v) \geq 0$ such that $ \psi ( u , v ) \psi ( u ^ { \prime } , v ^ { \prime } ) \leq \psi ( u , v ^ { \prime } ) \psi ( u ^ { \prime } , v )$ whenever $ u \leq u ^ { \prime }$ and $ v \leq v ^ { \prime }$ .On another hand, suppose $ C (u , v) = \psi \left( \phi ( u ) + \phi ( v ) \right)$  is Archimedean copula of $ ( X _ { 1 } , X _ { 2 } )$, $ ( X _ { 1 } , X _ { 2 } ) $ satisfies $RR_2$ if $\psi$ is log-concave (cf.\cite{karlin1980classes,barmalzan2020stochastic}). 
        \item When we take the independence copula $\psi ( t ) = e ^ { - t } $, we can get the independent case directly from the Theorem \ref{th11}.
\end{enumerate}

\end{remark}

\begin{remark}\label{r2}
Condition (ii) of Theorem \ref{th11} says $\bar { F } ( x ; e^{a_{i}}  )$ is decreasing and log-convex  in $a_{i} = \log\theta _ { i }$, for $i=1,2,...,n$. There are many specific models in the family of semi-parametric distributions that satisfy condition (ii) of Theorem \ref{th11}. Next,we present specific models that satisfy the condition (ii) of the Theorem and explain them further. 
\begin{enumerate}
  \item  Next, we consider the scale (SC) model, proportional hazard rates (PHR) model and the location (L) model to check the condition (ii) of Theorem \ref{th11}. See Table \ref{tab2} for details on these models.
\begin{table}[!h]
		\centering 
		\setlength{\belowcaptionskip}{0.4cm}
		\caption{Partial specific models that satisfy condition (ii) of Theorem \ref{th11}}
		\label{tab2}
		\begin{tabular}{cccc}
			\toprule 
			Specific model  & Survival function & Parameter name & Parameter scale  \\
			\midrule 
            SC model & $ \bar { F } ( x ; \theta ) = \bar { F } ( \theta  x )$ & scale & $\theta >0$\\
             PHR model & $ \bar { F } ( x ; \theta ) = \bar { F }^\theta (  x )$ & frailty & $\theta >0$ \\
             L model &  $ \bar { F } ( x ; \theta ) = \bar { F }   (x-\theta) $  &location &$-\infty<\theta<\infty $ \\
            
			\bottomrule
		\end{tabular}
	\end{table}
        \item The above three models all meet the Condition (ii) in Theorem \ref{th11}, and the Condition (ii) is equivalent to the baseline hazard rate function $ h _ { F } ( x )$ is decreasing in $x \in \mathbb{R}^+$(i.e.$ F $ has decreasing failure rate ($DFR$)). There are many well-known distributions that satisfy $DFR$ such as 
          \begin{enumerate}
          \item Burr (B) distribution function
          $$ F ( c,k;x ) = 1 - ( 1 + x ^ { c } ) ^ { - k },  x,c , k > 0;$$
          \item Exp. Weibull (EW) distribution function$$ F ( \alpha , \beta;x ) = ( 1 - e ^ { - x ^ { \alpha } } ) ^ { \beta } , x , \alpha , \beta > 0;$$
          \item Generalized Pareto (GP) distribution function $$ F ( \alpha ; x ) = 1 - ( 1 + \alpha x ) ^ { - 1 / \alpha },x,\alpha>0;$$  
          \item Generalized Gamma distribution (GG) with the density function
          $$ f ( p , q ; x ) = \frac { p x ^ { q - 1 } e ^ {- x ^ { p } } } { \Gamma ( q / p ) } , x ,p , q >0.$$
          \end{enumerate}
          It can be checked that the Burr distribution$(\alpha>0 , 0<\beta <1)$, EW distribution$(0<\alpha,\beta\leq 1)$, GP distribution$(\alpha > 0)$, GG distribution $(\alpha>0,0 < p, q < 1)$, LTW distribution$(0<\alpha\leq 1)$, PGW distribution$(c \leq k, c < 1)$ are DFR.(cf.\cite{barmalzan2020stochastic,hazra2017stochastic})
\end{enumerate}

\end{remark}
Next, we will give a numerical example to illustrate the feasibility of the derived results.
\begin{example}

First we consider the scale (SC) model in the family of semi-parametric distributions,  where the baseline distribution is the EW distribution$(\alpha=\beta=0.9)$, and according to remark 2, this model satisfies condition 2 of Theorem \ref{th11}. Take $(\theta_{1},\theta_{2},\theta_{3},\theta_{4},\theta_{5})=(0.12,0.28,0.51,0.62,0.73),(\theta_{1}^{\ast},\theta_{2}^{\ast},\theta_{3}^{\ast},
 \theta_{4}^{\ast},\\\theta_{5}^{\ast})=(0.21,0.42,0.73,0.89,0.92)$ and the generator $\psi(x)= \exp\left(\frac { 1 } { \theta } ( 1 - e ^ { x } ) \right),$ where $\theta =0.2$ corresponding to Gumble–Barnett copula. By simple mathematical calculations, we can find that $(\theta_{1},\theta_{2},\theta_{3},\\\theta_{4},\theta_{5}) \overset{p}{\succeq} (\theta_{1}^{\ast},\theta_{2}^{\ast},\theta_{3}^{\ast},\theta_{4}^{\ast},\theta_{5}^{\ast})$ and $\psi(x)$ is is log-concave. Therefore, the above conditions satisfy the three requirements of Theorem \ref{th11}.
 
 Based on the above provisions, explicit expressions of $ \bar { F }_{X_{2:5}} ( x ; \theta )$ and  $\bar { F }_{Y_{2:5}} ( x ; \theta^{\ast} )$ are as follows:
\begin{align*}
  \bar { F }_{X_{2:5}} ( x ; \boldsymbol\theta )
  &=\sum _ { i = 1 } ^ { 5 } \exp\left\{10\left\{{{1-\prod\limits_ {j\neq i }\left\{1-0.1\log\left[1-\left(1-e^{-( \theta_{j} x)^{0.9}}\right)^{0.9} \right]\right\}}}\right\}\right\} \\
   &\quad \quad-4\exp\left\{10\left\{{1-\prod\limits_ { i = 1 } ^ { 5 }\left\{1-0.1\log\left[1-\left(1-e^{-( \theta_{i} x)^{0.9}}\right)^{0.9} \right]\right\}}\right\}\right\},
\end{align*}
\begin{align*}
  \bar { F }_{Y_{2:5}} ( x ; \boldsymbol\theta^{\ast} )
  &=\sum _ { i = 1 } ^ { 5 } \exp\left\{10\left\{{{1-\prod\limits_ {j\neq i }\left\{1-0.1\log\left[1-\left(1-e^{-( \theta_{j}^{\ast} x)^{0.9}}\right)^{0.9} \right]\right\}}}\right\}\right\} \\
   &\quad \quad-4\exp\left\{10\left\{{1-\prod\limits_ { i = 1 } ^ { 5 }\left\{1-0.1\log\left[1-\left(1-e^{-( \theta_{i}^{\ast} x)^{0.9}}\right)^{0.9} \right]\right\}}\right\}\right\}.
\end{align*}

In Figure \ref{t1}, we plot the survival functions of $X_{2:5}$ and $Y_{2:5}$. We find that the image of $\bar{F}_{X_{2:5}}(x)$ is always higher than that of $\bar{F}_{Y_{2:5}}(x)$. That illustrates $\bar{F}_{X_{2:5}}(x) \geq \bar{F}_{Y_{2:5}}(x)$, for all $ x > 0.$ This confirms the result given in Theorem \ref{th11}.
    \begin{figure}[htbp]
    \centering
    \includegraphics{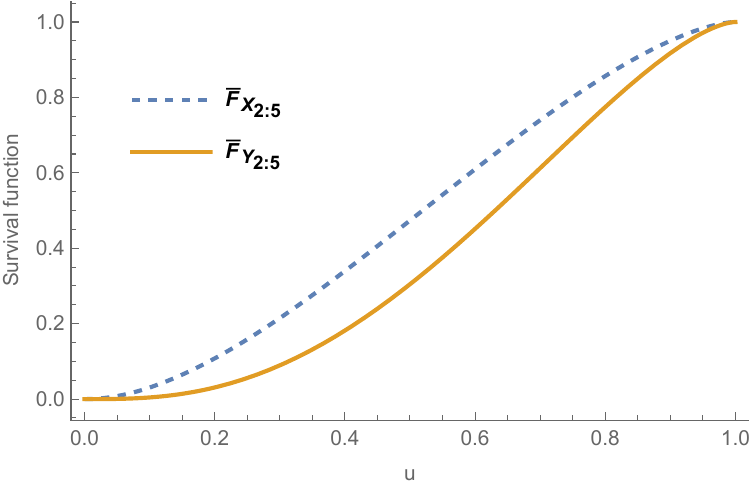} 
    \caption{Plots of $\bar{ F } _{{X} _ { 2:5 }} ( x )$ and  $\bar{ F } _{{Y} _ { 2:5 }} ( x ) $ , where $x=-\log u$ and $u \in (0, 1].$}\label{t1}
    \end{figure}

\end{example}
Logarithmic concavity of $\psi(x)$ is essential in Theorem \ref{th11}, and the following counterexample shows that log-concave cannot be replaced by log-convex.
\begin{example}

First we consider $\bar{F}  {  } ( x ; \theta _ { i } )=  { \rm e}^{-(\theta_{i}x)^{0.9}}$ and $\bar{F}  {  } ( x ; \theta _ { i } )=  { \rm e}^{-(\theta^{\ast}_{i}x)^{0.9}}$, $i=1,2,3,4,5$ in the family of semi-parametric distributions, Take $(\theta_{1},\theta_{2},\theta_{3},\theta_{4},\theta_{5})=(0.13,0.31,0.49,0.61,0.72),(\theta_{1}^{\ast},\theta_{2}^{\ast},\theta_{3}^{\ast},
 \theta_{4}^{\ast},\\\theta_{5}^{\ast})=(0.22,0.41,0.71,0.88,0.92)$ and the generator $\psi(x)= ( \theta x + 1 ) ^ { - 1 / \theta },$ where $\theta =10$ corresponding to Clayton copula. By simple mathematical calculations, we can find that $(\theta_{1},\theta_{2},\theta_{3},\theta_{4},\theta_{5}) \overset{p}{\succeq} (\theta_{1}^{\ast},\theta_{2}^{\ast},\theta_{3}^{\ast}, \theta_{4}^{\ast},\theta_{5}^{\ast})$ and $\psi(x)$ is is log-convex.
 
 Based on the above provisions, explicit expressions of $ \bar { F }_{X_{2:5}} ( x ; \theta )$ and  $\bar { F }_{Y_{2:5}} ( x ; \theta^{\ast} )$ are as follows:
$$ \bar { F }_{X_{2:5}} ( x ; \boldsymbol\theta ) =\sum _ { i = 1 } ^ { 5 } \left[ \sum _ { j \neq i } \left( { \rm e}^{-(\theta_{j}x)^{0.9}}-3\right)^{-\frac{1}{10}} \right]-4\sum _ { i = 1 } ^ { 5 }\left( { \rm e}^{-(\theta_{i}x)^{0.9}}-4\right)^{-\frac{1}{10}},$$
$$ \bar { F }_{Y_{2:5}} ( x ; \boldsymbol\theta^{\ast} ) =\sum _ { i = 1 } ^ { 5 } \left[ \sum _ { j \neq i } \left( { \rm e}^{-(\theta_{j}^{\ast}x)^{0.9}}-3\right)^{-\frac{1}{10}} \right]-4\sum _ { i = 1 } ^ { 5 }\left( { \rm e}^{-(\theta_{i}^{\ast}x)^{0.9}}-4\right)^{-\frac{1}{10}}.$$
In Figure \ref{t2}, we plot $\bar{F}_{X_{2:5}}(x)-\bar{F}_{Y_{2:5}}(x)$ and $y=0$, for all $x \in (0, 10]$. We find that the curve of $\bar{F}_{X_{2:5}}(x)-\bar{F}_{Y_{2:5}}(x)$  intersects curve $y=0$. That means it's neither $\bar{F}_{X_{2:5}}(x) \geq \bar{F}_{Y_{2:5}}(x)$ nor $\bar{F}_{X_{2:5}}(x) \leq \bar{F}_{Y_{2:5}}(x)$, for all $ x > 0.$ Therefore, log-concave cannot be replaced by log-convex.
    \begin{figure}[htbp]
    \centering
    \includegraphics{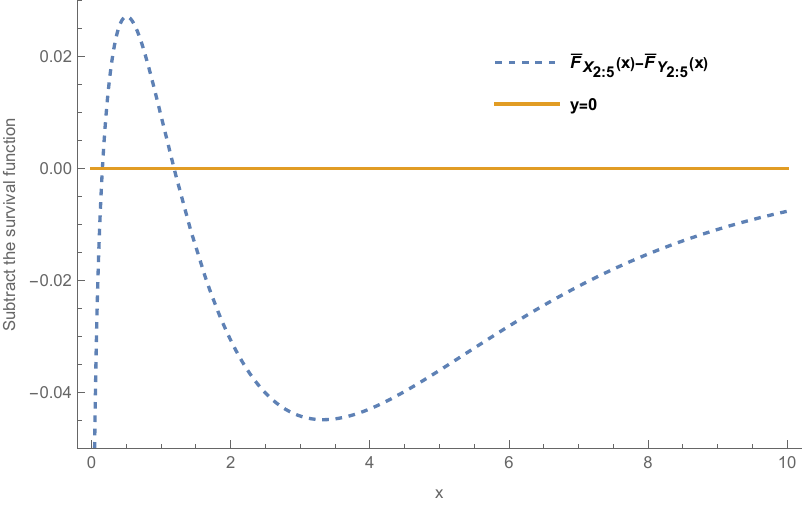} 
    \caption{Plots of $\bar{ F } _{{X} _ { 2:5 }} ( x )-\bar{ F } _{{Y} _ { 2:5 }} ( x ) $ , where $x \in (0, 10].$}\label{t2}
    \end{figure}

\end{example}
\begin{corollary}
By setting $ (\theta_{1}^{\ast},...,\theta_{n}^{\ast})=(\theta^{\ast},...,\theta^{\ast})$, where $ \theta^{\ast} \geq \left( \prod\limits _ { i = 1 } ^ { n } \theta _ { i } \right) ^ { 1 / n }$, one can see
that $\boldsymbol\theta \overset{p}{\succeq}  \boldsymbol\theta^{\ast}$. Thus, a lower bound for the survival function of $X _ { 2:n }$ is as follows: 
$$ \bar { F } _ { X _ { 2 : n } } ( x )\geq n \psi \left( ( n - 1 ) \phi \left( \bar { F } ( x ; \theta^{\ast} ) \right) \right) - ( n - 1 ) \psi \left( n \phi \left( \bar { F } ( x ; \theta^{\ast}  ) \right) \right) .$$ 

Therefore,  based on Theorem \ref{th11},  we can obtain reliability lower bounds for a fail-safe system with heterogeneous components from a fail-safe system with homogeneous components in the sense of the p-larger order.
\end{corollary}

Next, we study the usual stochastic order between two the second-order statistics when the parameter vectors verify the reciprocally majorization ordering. To do this, we consider two random vectors $\boldsymbol{X}$ and $\boldsymbol{Y}$ from a general semi-parametric family of distributions assembled by Archimedean distributional copula with common generator $\psi$, respectively. We present a general result for the second smallest order statistics whose survival functions are defined as in Theorem \ref{th11}.

\begin{theorem}\label{th2}
Let $\boldsymbol{X}$ and $\boldsymbol{Y}$ be two vectors of nonnegative dependent random variables with Archimed\\ean distributional copula of common generator $\psi.$ Assume  $X _ {i} \sim F  {  } ( x ; \theta _ { i } ) $ and $Y _ {i} \sim F  {  } ( x ; \theta _ { i }^{\ast} ) $ for $i=1,2,...,n$. If the following conditions hold:
\begin{enumerate}
  \item $\psi$ is log-convex{\rm{;}}
  \item $\bar { F } ( x ; \theta  )$ is increasing and log-convex  in $\theta${\rm{;}}
\end{enumerate}
Then, $\boldsymbol\theta \overset{rm}{\succeq}  \boldsymbol\theta^{\ast} $ implies $X _ { 2:n } \geq_ { s t } Y _ { 2 : n }$.
\end{theorem}
\begin{proof}
To obtain the desired result, it sufficient to show that
$$\bar { F }_{{X} _ { 2:n }} ( x ) \geq \bar{ F }_{{Y} _ { 2:n }} ( x ).$$

To do this, let us define
\begin{equation}
\ {h}_ { 2 }(\boldsymbol{\theta}) = \sum _ { i = 1 } ^ { n } \psi \left[ \sum _ { j \neq i } \phi ( \bar { F } ( x ; \theta _ { j }  ) ) \right] - ( n - 1 ) \psi \left[ \sum _ { i = 1 } ^ { n } \phi ( \bar{ F } ( x ; \theta _ { i } ) ) \right].
\end{equation}

Thus, from Lemma \ref{yl3}, it is required to prove that ${h}_ { 2 }(1/{a_{1}},...,1/{a_{n}})$ is increasing in $a_{i}$, for all $i = 1,..., n$, and it is Schur-convex in $\boldsymbol{a} = ( a _ { 1 } , \ldots , a _ { n } )$, where $a_{i} = 1/\theta _ { i }$, for $i = 1,...,n.$ Taking the partial derivative of ${h} _ { 2 }(1/\boldsymbol{{a}})$ with respect $a_{i}$, where $1/\boldsymbol{{a}}=(1/{a_{1}},...,1/{a_{n}})$,we obtained
\begin{align*}
  \frac { \partial {h} _ { 2 } ( 1/\boldsymbol{{a}}) } { \partial a _ { p } }
  &= \left\{\sum _ { i \neq p }  \psi^{\prime} \left[ \sum _ { j \neq i } \phi ( \bar { F } ( x ; 1/{a_{j}}  ) ) \right] - ( n - 1 ) \psi^{\prime} \left[ \sum _ { i = 1 } ^ { n } \phi ( \bar { F } ( x ; 1/{a_{i}} ) ) \right] \right\}\\
  &\quad \quad\times\left(\frac { \partial( \bar { F } ( x ; 1/{a_{p}}  ) ) } { \partial a _ { p } }\right) \frac{1}{\psi^\prime(\phi( \bar { F } ( x ; 1/{a_{p}}   ))}\\
  &\geq 0.
\end{align*}

In fact, because of the property of the composition of convex functions, we know that $\theta(a)=1/a$ is convex in a, and according to condition (ii) of the Theorem \ref{th2}, we have $\bar { F } ( x ; 1/{a_{i}}  )$  is increasing and log-convex in $ a_{i} $ and
for $1 \leq p \neq k \leq n $, let
\begin{align*}
  \eta _ {  } \left( z , x \right)
  &= ( n - 1 ) \psi ^ { \prime } \left[ \sum _ { i = 1 } ^ { n } \phi ( \bar{ F } ( x ; 1/{a_{i}} ) ) \right] - \sum _ { i \notin \left\{ p , k \right\} } \psi ^ { \prime } \left[ \sum _ { j \neq i } \phi ( \bar { F } ( x ; 1/{a_{j}} ) ) \right] \\
   &\quad \quad- \psi ^ { \prime } \left[ \phi ( \bar { F } ( x ; 1/{ z } ) ) + \sum _ { j \notin \left\{ p , k \right\} } \phi ( \bar{ F } ( x ; 1/{a_{j}} ) ) \right].
\end{align*}

According to Definition \ref{dycopula}, it holds that $\psi ^ { \prime } ( x ) \leq 0 $ for $x \geq 0$ and $\psi ^ { \prime }$ is increasing.
Since $ \phi ( \bar { F } ( x ; 1/{ z } ))$ is decreasing in $z > 0$, the function $\eta _ {  } ( z , x )$ is increasing in $z > 0$. In view of $\phi ( x ) \geq 0 $ for $x \in \left[ 0 , 1 \right] $, we have

$$\psi ^ { \prime } \left[ \sum _ { i = 1 } ^ { n } \phi \left( \bar { F } ( x ; 1/{a_{i}} ) \right) \right] \geq \psi ^ { \prime } \left[ \sum _ { j \neq i } \phi ( \bar { F } ( x ; 1/{a_{i}} ) ) \right], $$
and then
$$( n - 1 ) \psi ^ { \prime } \left[ \sum _ { i = 1 } ^ { n } \phi ( \bar { F } ( x ; 1/{a_{i}} ) ) \right] \geq \sum _ { i \neq p } \psi ^ { \prime } \left[ \sum _ { j \neq i } \phi ( \bar { F } ( x ; 1/{a_{i}} ) ) \right]. $$

As a result, for any $x \geq 0$, it holds that
\begin{align}
  \eta _ {  } ({a_{p}}, x )
  & = ( n - 1 ) \psi ^ { \prime } \left[ \sum _ { i = 1 } ^ { n } \phi ( \bar{ F } ( x ; 1/{a_{i}} ) ) \right] - \sum _ { i \notin \left\{ p , k \right\} } \psi ^ { \prime } \left[ \sum _ { j \neq i } \phi ( \bar { F } ( x ; 1/{a_{j}} ) ) \right] -  \sum _ { j \neq k} \phi ( \bar{ F } ( x ; 1/{a_{i}} ) ) \nonumber\\
  &= ( n - 1 ) \psi ^ { \prime } \left[ \sum _ { i = 1 } ^ { n } \phi ( \bar { F } ( x ; 1/{a_{i}} ) ) \right] - \sum _ { i \neq p } \psi ^ { \prime } \left[ \sum _ { j \neq i } \phi ( \bar { F } ( x ; 1/{a_{j}} ) ) \right] \nonumber \\
  & \geq 0.
  \nonumber
\end{align}

Therefore 
$$\sum _ { i \neq p }  \psi^{\prime} \left[ \sum _ { j \neq i } \phi ( \bar { F } ( x ; 1/{a_{j}}  ) ) \right] - ( n - 1 ) \psi^{\prime} \left[ \sum _ { i = 1 } ^ { n } \phi ( \bar { F } ( x ; 1/{a_{i}} ) ) \right] \leq 0,$$
and for $a _ { p } \geq a _ { q },$ we have $ \eta _ {  } ({a_{p}}, x ) \geq \eta _ {  } ({a_{q}}, x ) $, in other words
\begin{flalign}\label{e8}
  & \sum _ { i \neq p } \psi ^ { \prime } \left[ \sum _ { j \neq i } \phi ( \bar { F } ( x ; 1/{a_{j}} ) ) \right]-( n - 1 ) \psi^{\prime} \left[ \sum _ { i = 1 } ^ { n } \phi ( \bar { F } ( x ; 1/{a_{i}} ) ) \right]\nonumber \\
  &\leq \sum _ { i \neq q } \psi ^ { \prime } \left[ \sum _ { j \neq i } \phi ( \bar { F } ( x ; 1/{a_{j}} ) ) \right] - ( n - 1 ) \psi^{\prime} \left[ \sum _ { i = 1 } ^ { n } \phi ( \bar { F } ( x ; 1/{a_{i}} ) ) \right] \\
  &\leq 0.
  \nonumber
\end{flalign}

Since $\psi$ is log-convex and $\bar{F} ( x ; 1/{a_{i}} )$ is log-convex in $a_ { i }$ (or equivalently, $\frac { \partial( \bar { F } ( x ; 1/{a_{i}} ) ) } { \partial a _ { i } }\frac{1}{ \bar { F } ( x ; 1/{a_{i}}  ))}$ is increasing ), we have, for $a _ { p } \geq a _ { q } $,
\begin{equation}\label{e9}
\frac { \psi ( \phi ( \bar { F } ( x ;1/{a_{p}} ) ) } { \psi ^ { \prime } ( \phi ( \bar { F } ( x ; 1/{a_{p}} ) ) } \leq\frac { \psi ( \phi ( \bar { F } ( x , 1/{a_{q}} ) ) ) } { \psi ^ { \prime } ( \phi ( \bar { F } ( x ; 1/{a_{q}} ) ) } \leq 0,
\end{equation}
and
\begin{equation}\label{e10}
    {\frac { \partial( \bar { F } ( x ; 1/{a_{p}}   ) ) } { \partial a _ { p } }} \frac { 1 } { \bar { F } ( x ;  1/{a_{p}}  ) }\geq{\frac { \partial( \bar { F } ( x ; 1/{a_{q}}   ) ) } { \partial a _ { q } }}\frac { 1 } { \bar { F } ( x ; 1/{a_{q}}  ) } \geq 0.
\end{equation}

From (\ref{e9}) and (\ref{e10}) ,we get
\begin{equation}\label{e11}
 -\frac { \partial( \bar { F } ( x ; 1/{a_{p}}  ) ) } { \partial a _ { p } }\frac{1}{\psi^\prime(\phi( \bar { F } ( x ; 1/{a_{p}}  ))}\geq-\frac { \partial( \bar { F } ( x ; 1/{a_{q}}  ) ) } { \partial a _ { q } }\frac{1}{\psi^\prime(\phi( \bar { F } ( x ; 1/{a_{q}}  ))} \geq 0.
\end{equation}

For any $p \neq q$, we have
\begin{flalign}\label{e12}
  &\frac { \partial {h} _ { 2 } ( 1/\boldsymbol{{a}}) } { \partial a _ { p } } - \frac { \partial {h} _ { 2 } ( 1/\boldsymbol{{a}}) } { \partial a _ { q } }\nonumber\\
  &=\left\{\sum _ { i \neq p }  \psi^{\prime} \left[ \sum _ { j \neq i } \phi ( \bar { F } ( x ; 1/{a_{j}}  ) ) \right] - ( n - 1 ) \psi^{\prime} \left[ \sum _ { i = 1 } ^ { n } \phi ( \bar { F } ( x ; 1/{a_{i}} ) ) \right] \right\} \frac{\frac { \partial( \bar { F } ( x ; 1/{a_{p}}  ) ) } { \partial a _ { p } }}{\psi^\prime(\phi( \bar { F } ( x ; 1/{a_{p}}  ))}\\
  &-\left\{\sum _ { i \neq q }  \psi^{\prime} \left[ \sum _ { j \neq i } \phi ( \bar{ F } ( x ; 1/{a_{j}}  ) ) \right] - ( n - 1 ) \psi^{\prime} \left[ \sum _ { i = 1 } ^ { n } \phi ( \bar { F } ( x ; 1/{a_{i}} ) ) \right] \right\} \frac{\frac { \partial( \bar { F } ( x ; 1/{a_{q}}  ) ) } { \partial a _ { q } }}{\psi^\prime(\phi( \bar { F } ( x ; 1/{a_{q}}  ))}. \nonumber
\end{flalign}

On using (\ref{e8}) and (\ref{e11}) in (\ref{e12}), we get
$$\left( a _ { p } - a _ { q } \right)\left( \frac { \partial {h} _ { 2 } ( 1/\boldsymbol{{a}}) } { \partial a _ { p } } - \frac { \partial {h} _ { 2 } ( 1/\boldsymbol{{a}}) } { \partial a _ { q } } \right) \geq 0 .$$
Note that $\bar{ F } _{{X} _ { 2:n }}(x)$ is symmetric with respect to $a_{ i }$’s. Thus, by using Lemma \ref{yl2}, we get that $\bar{ F } _{{X} _ { 2:n }}(x)$ is Schur-convex in
$(a_{ 1 },...,a_{ n })$. Hence the result is proved.
\end{proof}

Theorem \ref{th2} shows that more heterogeneity among general semi-parametric family of distributions parameter vectors in the sense
of the reciprocally majorization ordering results in larger reliability under the usual stochastic order in reliability theory.

\begin{remark}
Condition (i) of Theorem \ref{th2} says $\psi$ is log-convex, which is not a complicated requirement. Next, we show that something satisfying condition (i) can be found, and give a further explanation of "$\psi$ is log-convex" . 
\begin{enumerate}
  \item There exist many well-known generator of Archimedean copulas which satisfy the log-convexity of $\psi$.In Table \ref{tab3} ,we shall give a summary  log-convex of generator of an Archimedean copula.
\begin{table}[!h]
		\centering 
		\setlength{\belowcaptionskip}{0.4cm}
		\caption{Summary log-convex of generator $\psi$ of an Archimedean copula}
		\label{tab3}
		\begin{tabular}{ccc}
			\toprule 
			Copula  & Generator & Parameter \\
			\midrule 
            Clayton & $( \theta t + 1 ) ^ { - 1 / \theta }$ & $\theta \in (0,\infty)$ \\
            Gumbel & $ \exp ( - t ^ { 1 / \theta } )$ & $\theta \in [1,\infty)$ \\
            Frank & $ - \frac { 1 } { \theta } \log ( 1 + ( e ^ { - \theta } - 1 ) e ^ { - t } )$ & $\theta \in ( 0 , \infty ) $ \\
            Ali-Mikhail-Haq & $ (1 - \theta)/ (\exp (t) - \theta)$  & $\theta \in [ 0 , 1 )$ \\
			\bottomrule
		\end{tabular}
	\end{table}
  \item The log-convexity of a generator $\psi$ of an Archimedean copula implies the property of decreasing left tail in sequence. (cf.\cite{you2014optimal}). 
  \item When we take the independence copula $\psi ( t ) = e ^ { - t } $, we can get the independent case directly from the Theorem \ref{th2}, and we won't say much about it here.
\end{enumerate}

\end{remark}

\begin{corollary}
By setting $ (\theta_{1}^{\ast},...,\theta_{n}^{\ast})=(\theta^{\ast},...,\theta^{\ast})$,where $ n\theta^{\ast} \leq\sum \limits_ { i = 1 } ^ { n } \theta _ { i }  $, one can see
that $\boldsymbol\theta \overset{rm}{\succeq}  \boldsymbol\theta^{\ast}$. Thus, a lower bound for the survival function of $X _ { 2:n }$ is as follows: 
$$ \bar { F } _ { X _ { 2 : n } } ( x )\geq n \psi \left( ( n - 1 ) \phi \left( \bar { F } ( x ; \theta^{\ast} ) \right) \right) - ( n - 1 ) \psi \left( n \phi \left( \bar { F } ( x ; \theta^{\ast}  ) \right) \right) .$$ 

Therefore,  based on Theorem \ref{th2},  we can obtain reliability lower bounds for a fail-safe system with heterogeneous components from a fail-safe system with homogeneous components in the sense of the reciprocally majorization ordering.
\end{corollary}
\section{Illustration with examples}\label{corollary}
Regarding the discussion of Remark \ref{r2}, the model naturally satisfies the requirements of the Theorem \ref{th11}, and due to the single parameter of the model and the simplicity of the treatment, we have not discussed it extensively. This section is more complex and flexible compared to the model considered in the previous section, and has been studied by a larger number of scholars.
 
 We apply Theorem \ref{th11} to modified proportional hazard rates scale (MPHRS) model and location scale model (LS) model,and get the following Proposition \ref{c1} and Proposition \ref{c2}. Through the conditions of Theorem \ref{th11} we only need to verify the decreasing property and logarithmic convexity of the model survival function.
\subsection{Modified Proportional Hazard Rates Scale model}
Next we examine several complex models in the family of semi-parametric distributions. \cite{das2021some} defined a new distribution, which is called modified proportional hazard rates scale  (MPHRS) model. The survival function of this model is given by
\begin{equation}\label{e7}
  \bar { F } ( x ; \alpha,\lambda ,\mu) =  \frac { \alpha  \left( \bar { F } ( x \mu ) \right) ^ { \lambda  } } { 1 - \bar { \alpha } \left( \bar { F } ( x \mu ) \right) ^ { \lambda  } }~~~~ where
   (x,\alpha,\lambda ,\mu>0) ,\bar{\alpha}=1-\alpha.
\end{equation} 

 It is worth noting that this model belongs to the family of semi-parametric semi-parametric distributions, and when the parameters in the model take a specific value, the model degenerates into our commonly ordinary model, as follows
 \begin{enumerate}
  \item MPHRS model reduce to  the proportional hazard rate (PHR) studied in \cite{zhao2011dispersive};\cite{fang2016stochastic}, if $\mu=1$ and $\alpha=1$;
  \item MPHRS model reduce to the scale (SC) model studied in \cite{li2016stochastic}, if $\alpha=1$ and $\lambda=1$,;
  \item MPHRS model reduce to  the scale proportional hazard (SPH) model studied in \cite{fang2018ordering}, if $\alpha=1$;
  \item MPHRS model reduce to  the modified proportional hazard rate (MPHR) model studied in \cite{balakrishnan2018modified};\cite{yan2023stochastic}, if $\mu=1$;
  \item MPHRS model reduce to the proportional odd (PO) model studied in \cite{panja2023dispersive}, if $\mu=1$ and $\lambda=1$.
 
\end{enumerate}
Now we apply Theorem \ref{th11} to MPHRS model, and get the following Corollary \ref{c1}. Through the conditions of Theorem \ref{th11} we only need to verify the decreasing property and logarithmic convexity of the model survival function.
\begin{proposition}\label{c1}
 Let $\boldsymbol{X}$ and $\boldsymbol{Y}$ be two vectors of dependent and heterogeneous random variables with Archimedean distributional copula with common generator $\psi.$ Assume  $\boldsymbol{X} \sim MPHRS(\boldsymbol\alpha,\boldsymbol\lambda ,\boldsymbol\mu,\bar { F })$ and $\boldsymbol{Y} \sim MPHRS(\boldsymbol\alpha^{\ast},\boldsymbol\lambda^{\ast} ,\boldsymbol\mu^{\ast},\bar { F })$, If $\psi$ is log-concave and the baseline hazard rate function $ x h _ { F } ( x )$ is decreasing in x, where $ \boldsymbol\alpha = \boldsymbol\alpha^{\ast} = \alpha \boldsymbol1 _ { n }$ with $ \alpha \in ( 0 , 1 ]$ and $ \boldsymbol\lambda = \boldsymbol\lambda^{\ast} = \lambda \boldsymbol1 _ { n }$.
 Then, $$\boldsymbol\mu \overset{p}{\succeq}  \boldsymbol\mu^{\ast} \Rightarrow X _ { 2:n } \geq _ { s t } Y _ { 2 : n } . $$

\end{proposition}
\begin{proof}
  We already know if the semi-parametric family of distributions follow the MPHRS model, then $\bar { F } ( x ; \alpha,\lambda ,\mu) =  \frac { \alpha  \left( \bar { F } ( x \mu ) \right) ^ { \lambda  } } { 1 - \bar { \alpha } \left( \bar { F } ( x \mu ) \right) ^ { \lambda  } }  $.
  To obtain the desired result, from Theorem \ref{th11}, it is sufficient to show that
$\bar { F } ( x ; e^{a_{i}}  )=\frac { \alpha  \left( \bar { F } ( xe^a_{i} ) \right) ^ { \lambda  } } { 1 - \bar { \alpha } \left( \bar { F } ( xe^a_{i}  ) \right) ^ { \lambda  } }   $ is decreasing and log-convex with respect to $a_{i} = \log\mu _ { i }$, for $i=1,2,...,n.$ We first show that the $ { F } ( x ; e^{a_{i}}) =  \frac { 1-\alpha  \left(1-  { F } ( x e^{a_{i}} ) \right) ^ { \lambda  } } { 1 - \bar { \alpha } \left(1-  { F } ( x e^{a_{i}} ) \right) ^ { \lambda  } }  $ is increasing with respect to $a_{i}$, for $i = 1,..., n$.
After some simple calculations, the partial derivative of the above function with respect to $a_{i}$, we have
\begin{align}
  \frac { \partial  { F } ( x ; e^{a_{i}}  ) } { \partial a _ { i } }
  & = \alpha \lambda x e ^ { a _ { i } } h_{F} ( x e ^ { a _ { i } } ) \frac { \left( 1 - F ( x e ^ { a _ { i } } ) \right) ^ { \lambda } } { ( 1 - \bar { \alpha } \left( 1 - F ( x e ^ { a _ { i } } ) \right) ^ { \lambda } ) ^ { 2 } } \nonumber\\
  & = \frac { \lambda x e ^ { a _ { i } } h _ { F } ( x e ^ { a _ { i } } ) } { 1 - \bar { \alpha } \left( 1 - F ( x e ^ { a _ { i } } ) \right) ^ { \lambda } } \bar { F } ( x ; e ^ { a _ { i } } ) \nonumber \\
  & \geq 0 .
  \nonumber
\end{align}

In other words, $ { F } ( x ; e^{a_{i}})$ is increasing with respect to $a_{i}$, for $i = 1,..., n$, which means that $ \bar{ F } ( x ; e^{a_{i}})$ is decreasing in $a_{i}$, for $i = 1,..., n$.
Besides, we have also
$$\frac { \partial \log\bar { F } ( x ; e^{a_{i}}  ) } { \partial a _ { i } }=- \frac { \lambda x e ^ { a _ { i } } h _ { F } ( x e ^ { a _ { i } } ) } { 1 - \bar { \alpha } \left( \bar { F } ( x e ^ { a _ { i } } ) \right) ^ { \lambda } } .$$ 

Now, the second order partial derivative of $\log\bar { F } ( x ; e^{a_{i}}  )$ with respect to $a_{i}$ is obtained as
\begin{align}
  \frac { \partial^2  \log{ F } ( x ; e^{a_{i}}  ) } { \partial {a _ { i }}^2 }
  & \overset{\rm sgn}{=}  - \left\{\left [ 1 - \bar { \alpha } \left( \bar { F } ( x e ^ { a _ { i } }) \right)  ^ { \lambda } \right] \frac { \partial ( x e ^ { a _ { i } } h _ { F } ( x e ^ { a _ { i } } ) ) } { \partial a _ { i } } - \bar { \alpha } \lambda \left( x e ^ { a _ { i } } h _ { F } ( x e ^ { a _ { i } } ) \right ) ^ { 2 } \left( \bar { F } ( x e ^ { a _ { i } } ) \right) ^ { \lambda } \right \} \nonumber\\
  & \geq 0, 
  \nonumber
\end{align}
which holds because  $ u h _ { F } ( u )$ is  decreasing in $u$ $(u=x e ^ { a _ { i } })$ and $ \alpha \in ( 0 , 1 ]$.
This illustrates $\bar { F } ( x ; e^{a_{i}}  )$ is log-convex respect to $a_{i} = \log\mu _ { i }$, for $i=1,2,...,n.$ Therefore, we get the desired result in the inference.

\end{proof}
\begin{remark}
The expression  $ u h _ { F } ( u )$  is known as the proportional failure rate (also known as the generalized failure rate) (cf.\cite{righter2009intrinsic}). $F$ is an increasing(decreasing) proportional failure rate IPFR (DPFR) distribution if $ u h _ { F } ( u )$ is increasing (decreasing) in u. The monotonicity
of these functions and their relationship with other notions of aging have been studied
in \cite{oliveira2015proportional}.
\end{remark}
\subsection{Location Scale model}
A random variable $X$ is said to follow the location–scale family, written as $X \sim LS(\lambda, \theta )$, if its distribution function is represented as$$ F _ { X } ( x ; \lambda , \theta ) = F \left(   \theta\left(x - \lambda\right)   \right) , x > \lambda ,$$where $\lambda (\in \mathbb{R})$ and $\theta (> 0)$ are the location (or threshold) and the scale parameters, respectively. There are very many distribution-satisfying LS models, and the interested reader is referred to the \cite{hazra2018stochastic} and \cite{barmalzan2023stochastic}.
Now we apply Theorem \ref{th11} to LS model, and get the following Corollary 2.
\begin{proposition}\label{c2}
 Let $\boldsymbol{X}$ and $\boldsymbol{Y}$ be two vectors of dependent and heterogeneous random variables with Archimedean distributional copula with common generator $\psi.$ Assume  $\boldsymbol{X} \sim LS(\boldsymbol\theta,\boldsymbol\lambda ,\bar { F })$ and $\boldsymbol{Y} \sim LS(\boldsymbol\theta^{\ast},\boldsymbol\lambda^{\ast} ,\bar { F })$, If $\psi$ is log-concave and the baseline hazard rate function $ x h _ { F } ( x )$ is decreasing in x.
 Then, $$\boldsymbol\theta \overset{p}{\succeq}  \boldsymbol\theta^{\ast} \Rightarrow X _ { 2:n } \geq _ { s t } Y _ { 2 : n } . $$

\end{proposition}
\begin{proof}
  We already know if the semi-parametric family of distributions follow the LS model, then $\bar { F } ( x ; \theta,\lambda) =  \bar { F } (  \theta(x-\lambda) ) $.
  To obtain the desired result, from Theorem \ref{th11}, it is sufficient to show that
$\bar { F } ( x ; e^{a_{i}} ) =  \bar { F } \left(  e^{a_{i}} (x-\lambda) \right)$ is decreasing and log-convex with respect to $a_{i} = \log\theta _ { i }$, for $i=1,2,...,n.$ After some simple calculations, the partial derivative of the above function with respect to $a_{i}$ ,we have
$$\frac { \partial  { \bar{F} } ( x ; e^{a_{i}}  ) } { \partial a _ { i } }= -e^{a_{i}}(x-\lambda)h  _ { F } ( e^{a_{i}}(x-\lambda) ){ \bar{F} } ( x ; e^{a_{i}}  )\leq0 ~~where~~ x\geq\lambda ,$$
which means that $ \bar{ F } ( x ; e^{a_{i}})$ is decreasing in $a_{i}$, for $i = 1,..., n$.
Besides, we have also
$$\frac { \partial \log\bar { F } ( x ; e^{a_{i}}  ) } { \partial a _ { i } }=-e^{a_{i}}(x-\lambda)h  _ { F } ( e^{a_{i}}(x-\lambda) ) .$$ 

Now, the second order partial derivative of $\log\bar { F } ( x ; e^{a_{i}}  )$ with respect to $a_{i}$ is obtained as
\begin{align}
  \frac { \partial^2  \log{ F } ( x ; e^{a_{i}}  ) } { \partial {a _ { i }}^2 }
  & =  -  \frac { \partial ( x e ^ { a _ { i } } h _ { F } ( x e ^ { a _ { i } } ) ) } { \partial a _ { i } } \nonumber\\
  & \geq 0 ,
  \nonumber
\end{align}
which holds because  $ t h _ { F } ( t )$ is  decreasing in $t$ $(t=e^{a_{i}}(x-\lambda))$ and $ x\geq\lambda$.
This illustrates $\bar { F } ( x ; e^{a_{i}}  )$ is log-convex respect to $a_{i} = \log\mu _ { i }$, for $i=1,2,...,n.$ Therefore, we get the desired result in the inference.

\end{proof}
\section{A real application of improving cable tensile strength in high voltage transmission network system}\label{data}
In this section, we consider a set of real datasets, which we use for reliability analysis along with the available results..

Normally, a cable consists of multiple wires. However, due to factors such as long-term use and harsh external environment, the continuous and efficient transmission performance of a cable becomes worse as the tensile strength of the cable wires decreases. Therefore, uniform cables with high tensile strength are essential for high-voltage transmission networks.

In a typical cable design, each cable consists of 12 wires. In order to study the tensile strength of the cables, we considered a data set in which samples of each wire in 9 cables were tested for tensile strength(cf.\cite{hand1993handbook}). For the improvement of sustained and efficient transmission of high voltage transmission network, we have considered the fail-safe system and we want to improve the reliability by our design to maintain the normal transmission for a long time. For this purpose, we try the following methods for reliability analysis.

\begin{enumerate}
  \item Select an appropriate fitting distribution.
  
  For the purpose of selecting the baseline distribution of cable strengths formed by different conductors, it is first assumed that the tensile strength life of the cable obeys some specific distribution (e.g.Exponential, Gamma,  Weibull, Burr), and then we use the  Akaike’s Information Criterion (AIC) and Bayesian Information Criterion (BIC) for selecting the best-fitting distribution function from the alternative distributions(cf.\cite{massey1951kolmogorov},\cite{rockette1974maximum}). From Table \ref{tab4}, we know that weibull is the best fitting distribution $ \bar { F } ( x ) = e ^ { - ( \frac { x } { a } ) ^ { b} }$, and the parameters are estimated as $\hat{a}=67.739,\hat{b}=341.65$. 
  
  \begin{table}[!h]
		\centering 
		\setlength{\belowcaptionskip}{0.4cm}
		\caption{Goodness-of-fit criteria.}
		\label{tab4}
		\begin{tabular}{ccccc}
			\toprule 
			Criteria  &  Exponential & Gamma & Weibull & Burr \\
			\midrule 
             AIC & 124.8621 & 62.1677 & 60.5968& 62.7108\\
             BIC & 125.0593& 62.5621 & 60.9913 & 63.3025\\
            
			\bottomrule
		\end{tabular}
	\end{table}

   Let us consider two sets of two four- component cable network systems made of \#wire 1, 3, 7, 8 and \#wire 2, 4, 5, 9 respectively, which follow the fail-safe system$({X} _ { 2:4 },{Y} _ { 2:4 })$. It is reasonable to assume here that all components follow the semi-parametric family of distributions, where the baseline distribution function corresponds to the parameter values. By a simple calculation, we can verify that the establishment $(341.3373,341.6547,342.1098,343.2238) \overset{p}{\succeq} (340.3877,342.6283,344.6258,345.1538)$.
   
   \item Fitting a proper Archimedean Copula.

    Next we used rank-based nonparametric estimation, and the Cramér-von Mises test statistic, as detailed in \cite{fermanian2005goodness,genest2008validity}. We selected the best Copula from among the alternative Copula. Table \ref{tab5} shows the estimated parameters, test statistics, and p-values for these three Copula. It can be seen that the Clayton Copula connective is the best candidate.
\begin{table}[!h]
		\centering 
		\setlength{\belowcaptionskip}{0.4cm}
		\caption{Goodness-of-fit test on copulas.}
		\label{tab5}
		\begin{tabular}{cccc}
			\toprule 
			Copula  &   Parameters & Test statistics & p-value  \\
			\midrule 
              Clayton & 1.0822& 0.0637 & 0.3358\\
             Gumbel & 1.5104& 0.0804 & 0.3010 \\
             Frank  & 2.9428& 0.0873 & 0.3308 \\
            
			\bottomrule
		\end{tabular}
	\end{table}

 \item Optimize power circuits by selecting cables with the greatest system life corresponding to tensile strength.

   The conditions of Theorem \ref{th11} are well satisfied and we next plot the survival functions of $X_{2:4}$ and $Y_{2:4}$. In Figure \ref{t3}, we find that the image of $\bar{F}_{X_{2:4}}(x)$ is always higher than that of $\bar{F}_{Y_{2:4}}(x)$. Therefore, in manufacturing cables, we prefer that the 12 wires of a cable include \#wire 1, 3, 7, 8. This results in cables with the best tensile strength, and these cables can operate more efficiently and durably in high-voltage transmission network systems.
\end{enumerate}
\begin{figure}[htbp]
    \centering
    \includegraphics{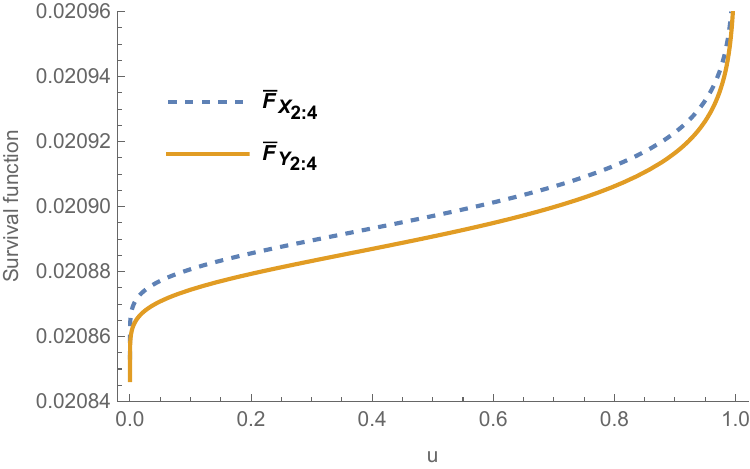} 
    \caption{Plots of $\bar{ F } _{{X} _ { 2:4 }} ( x )$ and  $\bar{ F } _{{Y} _ { 2:4 }} ( x ) $ , where $x=-\log u$ and $u \in (0, 1].$}\label{t3}
\end{figure}

\section{Concluding remarks}\label{conclud}
In this article, we consider two samples dependent and heterogeneous in a family of semi-parametric distributions. Sufficient conditions are established to randomly compare the second-order statistics of the p-larger and the reciprocally majorization of parameter vectors in two samples in the sense of  the usual stochastic order.

It is noteworthy that our present results only provide sufficient conditions in the sense of  the usual stochastic order for a randomized comparison of two samples dependent and heterogeneous in a family of semi-parametric distributions. In the future studies, we will continue to extend the results to the  hazard rate order and reversed hazard rate order with dependent and heterogeneous observations. In addition, we also consider more complex variants, for example, convex, star and dispersive orderings, etc. However, due to the complexity of the mathematical formula and the difficulty of calculation, these interesting problems remain open and merit further discussion.


\begin{funding}
Author was supported by by the National Natural Science Foundation of China Grant-12361060, College Teachers Innovation
Foundation Project of Gansu Provincial Education Department Grant-2024A-002.
\end{funding}



\bibliographystyle{imsart-nameyear} 
\bibliography{bmyref.bib}       


\end{document}